\documentclass[11pt]{article}

\usepackage[T1]{fontenc}
\usepackage{amsthm, amssymb, amsmath, mathtools, enumerate, graphicx, tikz, fullpage}
\usepackage[capitalize]{cleveref}

\newtheorem{theorem}{Theorem}[section]
\newtheorem{lemma}[theorem]{Lemma}

\newtheorem{corollary}[theorem]{Corollary}
\newtheorem{claim}[theorem]{Claim}
\newtheorem{conjecture}[theorem]{Conjecture}

\theoremstyle{definition}
\newtheorem{definition}[theorem]{Definition}
\newtheorem{notation}[theorem]{Notation}

\theoremstyle{remark}
\newtheorem{remark}[theorem]{Remark}
\newtheorem{example}[theorem]{Example}

\newcommand{\C}{\mathcal{C}}
\newcommand{\K}{\mathcal{K}}

\newcommand{\F}{\mathcal{F}}
\newcommand{\M}{\mathcal{M}}
\newcommand{\N}{\mathcal{N}}
\newcommand{\Q}{\mathcal{Q}}
\newcommand{\cs}{\mathcal{S}}

\usepackage{amscd}
\usepackage{setspace}
\usepackage{comment}
\usepackage{enumerate}
\theoremstyle{plain}
\numberwithin{equation}{section}

\theoremstyle{definition}

\theoremstyle{remark}

\newcommand{\D}{\mathcal{D}}
\newcommand{\R}{\mathcal{R}}

\newcommand{\cp}{\mathcal{P}}

\newcommand{\G}{\mathcal{G}}

\newcommand{\ch}{\mathcal{H}}

\newcommand{\tb}{\tilde{B}}

\newcommand{\T}{\mathcal{T}}

\newcommand{\vc}{V^\circ}

\newcommand{\cb}{\mathcal B}

\newcommand{\es}{\mathcal{E}}
\newcommand{\tes}{\tilde{\mathcal{E}}}

\newcommand{\ines}{\mathcal{IE}}
\newcommand{\tines}{\tilde{\mathcal{IE}}}

\title{Cooperative conditions for the existence of rainbow matchings}

\begin{document}

\author{Ron Aharoni
\footnote{Department of Mathematics, Technion, Haifa, Israel 
and MIPT, Dolgoprudny, Russia.
E-mail: \texttt{raharoni@gmail.com}.
}
\and 
Joseph Briggs
\footnote{
Department of Mathematics and Statistics,
Auburn University, AL, USA.
E-mail: \texttt{jgb0059@auburn.edu}.
}
\and
Minho Cho
\footnote{
Department of Mathematical Sciences, KAIST, Daejeon, Republic of Korea.
E-mail: \texttt{mhc0925@kaist.ac.kr}.
}
\and
Jinha Kim
\footnote{Corresponding Author. 
Discrete Mathematics Group, Institute for Basic Science (IBS), Daejeon, Republic of Korea.
E-mail: \texttt{jinhakim@ibs.re.kr}.
}
}


\date\today

\maketitle

\begin{abstract}
Let $k>1$, and let $\mathcal{F}$ be a  family of   $2n+k-3$ non-empty sets of edges in a bipartite graph. If the union of every $k$ members of $\mathcal{F}$  contains a matching of size $n$,  then there exists an $\mathcal{F}$-rainbow matching of size $n$. Replacing $2n+k-3$ by $2n+k-2$, the result is true also for $k=1$, and it can be proved (for all $k$) both topologically and by a relatively simple combinatorial argument. The main effort is in gaining the last $1$, which makes the result sharp. 
\end{abstract}

\section{Introduction}
Throughout the paper, ``family'' means ``multiset'', meaning that elements may repeat. 
To differentiate the notation, we use round brackets for  families, and (as usual) curly brackets for sets. 
For a family $\F$, we write $\F \setminus \{F\}$ and $\F \cup \{F\}$ in the family sense. That is, $\F\setminus\{F\}$ contains one less copy of $F$ than $\F$ if $F \in \F$, and $\F \cup \{F\}$ contains one more copy of $F$ than $\F$.

Given a family $\cs=(S_1, \ldots,S_m)$ of sets, an $\cs$-{\em rainbow set} is
the image of a partial choice function of $\cs$. So, it is a set $\{x_{i_j} \mid j\le k\}$, where $1 \le i_1< \ldots <i_k\le m$ and $x_{i_j}\in S_{i_j}$.  

A {\em complex} is a closed down hypergraph, meaning that any subset of any edge is an edge. 
The injectivity  - at most one element from every set $S_i$ - 
is a ``smallness'' condition, in the sense that the set of injective choices is a  complex.  Hence statements of interest are of the form ``there exists a large rainbow set satisfying certain conditions (like being a matching)''. The classical theorem  of this type is Hall's marriage theorem. 

Below, again, $\cs=(S_1, \ldots,S_m)$ is a family of sets. 
For a set $I \subseteq [m]$, let $\cs_I=\bigcup_{i \in I}S_i$.

\begin{theorem}
    If $|\cs_J| \ge |J| $ for every $J \subseteq [m]$ then there is a full rainbow set, that is, a rainbow set of size $m$. 
\end{theorem}

Another well-known rainbow result is Drisko's theorem, on rainbow matchings. The following slightly more general version of the original theorem was proved in \cite{ab}:

\begin{theorem}\cite{drisko}\label{drisko}
    $2n-1$ matchings in a bipartite graph, of size $n$ each, have a rainbow matching of size $n$. 
\end{theorem}

There is a conspicuous difference between the two theorems: in  the first the condition is ``cooperative'', namely it is on subfamilies of $\cs$, whereas in the second it is on singletons - each $S_i$ is assumed to be large by itself. On the other hand, there is a condition on  the number of the sets $S_i$. 

\subsection{A cooperative version of the Kalai-Meshulam theorem}
 A complex $\C$ is said to be $d$-{\em Leray} if $\tilde{H}_k(\C[S])=0$ for all $S \subseteq V$ and all $k \ge d$ ($\tilde{H}_k$ is the reduced $k$-th homology group). Let $\lambda(\C)$ be the smallest number $d$ such that 
$\C$ is $d$-Leray.

A basic result in this direction is a theorem of Kalai and Meshulam \cite{km}:

\begin{theorem}\label{km}
Let $\M$ and $\C$ be a matroid and a complex, respectively, on the same ground set. If $\lambda(lk_\C(S)) < rank_\M(V\setminus S)$ for every $S \in \C$ then $\M \setminus \C \neq \emptyset$.
\end{theorem}
Here $lk_\C(S) = \{T \subseteq V \setminus S \mid S\cup T \in \C\}$. The theorem above is a re-formulation of Theorem 1.6 in \cite{km}.

The following was proved in \cite{km2}:

\begin{theorem}\label{link}
For any complex $\C$ and set $S \in \C$, $\lambda(lk_\C(S))\le \lambda(\C)$.
\end{theorem}

Theorems \ref{km} and \ref{link}, combined, yield the following:

\begin{theorem}\label{partitionmatroidkm}
If $\lambda(\C)\le d$ and $\cs=(S_1, \ldots ,S_{d+k})$ is a family of  subsets of $V(\C)$  satisfying $\cs_I\not \in \C$ whenever $I\subseteq [d+k]$ is of size $k$, then there exists an $\cs$-rainbow non-$\C$ set.
\end{theorem}

\begin{proof}
By duplicating vertices, if necessary (a vertex having a distinct copy for every set $S_i$ it belongs to), we may assume that the sets $S_i$ are disjoint. 
Let $\M$ be the partition matroid defined by the sets $S_i$. By Theorems \ref{link} and \ref{km} 
it suffices to show that if $S \in \C$ then $rank_\M(V \setminus S)> d$. This follows from the condition  $\cs_I \not \in \C~~(|I| \ge k)$ and the fact that $rank_\M(A)=|\{i : A \cap S_i \neq \emptyset\}|$.  
\end{proof}

This is a ``cooperative'' version of the Kalai-Meshulam theorem, namely many sets join forces to contain a set not belonging to $\C$. 

\subsection{A cooperative version of Theorem \ref{drisko}}
For a set $F$ of edges we denote by $\nu(F)$ the maximal size of a matching in $F$. For a family $\F=(F_1,\ldots ,F_m)$ of sets of edges, we denote by $\nu_R(\F)$ the maximal size of an $\F$-rainbow matching.

Let  $t$ be an integer, and let $n\le t$. Let $\C$ be the complex consisting of all $F \subseteq E(K_{t,t})$, satisfying $\nu(F)<n$. In \cite{ahj} it was shown that $\lambda(\C)$ 
$\leq 2n-2$. 
Together with Theorem \ref{partitionmatroidkm} this yields:

\begin{theorem}\label{coopdrisko}
$2n+k-2$ sets of edges in a bipartite graph, the union of any $k$ of which contains a matching of size $n$, have a rainbow matching of size $n$.
\end{theorem}
\begin{notation}
We write $(m, k, n)\to_\cb q$ for the statement ``every $m$ nonempty sets of edges in a bipartite graph, the union of every $k$ of which contains a matching of size $n$, have a rainbow matching of size $q$''.  
\end{notation}
In this notation, the theorem says that $(2n+k-2, k, n) \to_\cb n$. The case $k=1$ is Theorem \ref{drisko}. 
The main result of this paper 
is that for $k>1$ this can be improved  by $1$, thereby obtaining a sharp bound. 

\begin{theorem}\label{main}
$(2n+k-3, k,n)\to_\cb n$ whenever $1< k\le n$. 
\end{theorem}

The sharpness of this result, namely the fact that
$(2n+k-4,k,n)\not \to_\cb n$ for any $k$, is given by the following example. In $C_{2n}$, take the odd  edges matching   repeated $n-1$ times, the even edges matching repeated $n-2$ times, and a singleton set, consisting of an even edge, repeated $k-1$ times. Explicitly: 
\begin{example}
Consider a complete bipartite graph $K_{n,n}$ with sides  $\{a_1,a_2,\ldots,a_n\}$ and $\{b_1,b_2,\ldots,b_n\}$.
Let
\[ S_i=
\begin{cases}
\{a_1b_1,a_2b_2,\ldots,a_nb_n\} &\text{ if }i \in [n-1],\\
\{a_1b_2,a_2b_3,\ldots,a_{n-1}b_n, a_nb_1\} & \text{ if } i \in [2n-3]\setminus[n-1],\\
\{a_1b_2\} &\text{ if } i \in [2n+k-4]\setminus[2n-3].
\end{cases}
\]
\end{example}
Let $\mathcal{S}=(S_i \mid i=1,\ldots,2n+k-4)$.
Then for any $I \subseteq[2n+k-4]$ with $|I| \geq k$, $\nu(\cs_I)\geq n$, and  $\nu_R(\mathcal{S})<n$.

\begin{remark}
After our result was obtained, Holmsen and Lee \cite{seunghun-holmsen} gave a topological proof of Theorem \ref{main}, using a strong version of Theorem \ref{km}. Their result is a somewhat stronger version of Theorem \ref{main}.
\end{remark}

\subsection{Cooperative versions of Colorful Caratheodory}

Part of the motivation for  Theorem \ref{main} comes from the existence of cooperative versions of a famous rainbow result - B\'ar\'any's Colorful Caratheodory theorem \cite{Ba}. In fact, as we shall see  below (first proof of Theorem \ref{thm: simplecase}), the affinity is not merely formal. 
Theorem \ref{coopdrisko} follows from a cooperative  version of Colorful Caratheodory. 

 Wegner \cite{wegner} noted that the complex  $\C$ of sets of vectors in $\mathbb{R}^d$ not containing a given vector $v$ in their convex hull satisfies $\lambda(\C)=d$. Similarly, the complex  $\D$ of sets not containing $v$ in their cone (set of non-negative combinations) satisfies $\lambda(\D)=d-1$. This, together with Theorem \ref{partitionmatroidkm}, yields:

\begin{theorem}\label{coopboth}
Let $v \in \mathbb{R}^d$.
\begin{enumerate}
    \item If $\cs=(S_1, \ldots ,S_{d+k})$ is a family of subsets of $\mathbb{R}^d$ such that $v \in conv(\cs_K)$ for every $K \subseteq [d+k]$ of size $k$, then there exists an $\cs$-rainbow set $S$ such that $v \in conv(S)$. 
    
    \item If $\cs=(S_1, \ldots ,S_{d+k-1})$ is a family of subsets of $\mathbb{R}^d$ such that $v \in cone(\cs_K)$ for every $K \subseteq [d+k-1]$ of size $k$, then there exists an $\cs$-rainbow set $S$ such that $v \in cone(S)$. 
\end{enumerate}

\end{theorem}
The case $k=2$ of part (1) of the theorem was strengthened by Holmsen-Pach-Tverberg \cite{hpt} and Arocha et.al. \cite{abbfm}:

\begin{theorem}\label{coopbar}
If $S_1, \ldots ,S_{d+1}$ are non-empty sets in $\mathbb{R}^d$, and $v \in conv(S_i \cup S_j)$ whenever $1\le i<j\le d+1$, then there is a rainbow set $S$ with $v \in conv(S)$.
\end{theorem}

Holmsen \cite{holmsen} gave a topological proof of this result, using a notion he called ``near $d$-Lerayness'', which means that $lk_\C(S)$ is $d$-Leray for every non-empty $S \in \C$. The same argument can be used to prove the analogous strengthening for all $k>1$:

\begin{theorem}
Let $k>1$, and let $\cs=(S_1, \ldots ,S_{d+k-1})$ be a family of non-empty sets in $\mathbb{R}^d$,  such that every $k$
of them contain $v$ in the convex hull of their union. Then there is an $\cs$-rainbow set containing $v$ in its convex hull. 
\end{theorem}

The analogous strengthening of part (2) of Theorem \ref{coopboth} is false, as witnessed by simple counterexamples.

\begin{example}
Let $v_1, \ldots, v_{d+1}$ be the vertices of a $d$-dimensional simplex $\sigma \subseteq \mathbb{R}^d$ whose barycenter is the origin. Let $v$ be the barycenter of face $\{v_1, \ldots, v_d\}$ of $\sigma$. Consider the family $\cs = (S_1,\ldots,S_{d+k-2})$ of non-empty sets in $\mathbb{R}^d$, where $S_i = \{v_1, \ldots, v_d\}$ for $1 \leq i \leq d-1$ and $S_j = \{v_{d+1}\}$ for $d \leq j \leq d+k-2$. Among any $k$ sets in $\cs$, at least one is $S_i$ for some $1 \leq i \leq d-1$, hence the convex cone spanned by their union contains $v$. However, there is no $\cs$-rainbow set $S$ such that $v \in cone(S)$.
\end{example}

\section{Rainbow paths}
The proof of Theorem \ref{main} is based on a combinatorial proof of the result $(2n+k-2, k, n) \to_\cb n$, and analysis of the extreme case. This proof, in turn, uses a lemma on rainbow paths in networks. To get the extra $1$ we analyze the extreme cases of that lemma. The analysis uses ideas from an  analogous lemma in \cite{akz}, which is the case $k=1$. But apart from a higher level of complexity, there is the difference that for $k>1$ the analysis leads to an improvement of $1$ in the theorem - which was not the case for $k=1$.  

A {\em network} is a triple $\N=(D,s,t)$, where $D$ is a digraph, and $s,t$ are two special vertices in it, called {\em source} and {\em target}. We assume that there are no edges going out of $t$ or into $s$.
We write $V(\N)$ for $V(D)$. The set $V(\N) \setminus \{s,t\}$ is denoted by $V^\circ(\N)$, and its elements are called  ``inner vertices''. For an $s-t$ path $P$ let  $\vc(P)=\vc(\N) \cap V(P)$. 
Two $s-t$ paths $P,Q$ are said to be {\em internally disjoint} if $V^\circ(P) \cap \vc(Q)=\emptyset$. 

For an $s-t$ path $Q$ let $B(Q)$ be the set of backward edges on $Q$, namely those directed edges $pq$ where $p,q \in V(Q)$ and $q$ precedes $p$ on $Q$.
Let $s_Q$ be the vertex   following $s$ in $Q$, and $t_Q$  the vertex  preceding $t$ in $Q$. Define $U(Q)=\{v s_Q \mid v \in V^\circ(\N) \setminus V(Q) \} \cup \{t_Q u \mid u \in  V^\circ(\N) \setminus V(Q)\}$. (``$U$" stands for ``useless'', since such edges cannot be used as shortcuts - this will be clarified below). 

We shall borrow a term - ``regimented'' - from \cite{akz}, but its use is a bit different here.  

\begin{definition}\label{reg}
Let  $\F$ be a family  of sets of edges in $\N$. A {\em regimentation} of $\F$ is a pair $\R=(\Q=\Q(\R),I=I(\R))$, where $\Q$ is a set of internally disjoint $s-t$ paths, and $I$ is a function from a subset $\es=\es(\R)$ of $\F$ (the ``essential'' sets)  onto $\Q$, satisfying the following conditions:  

\begin{enumerate}
\item $\bigcup_{Q \in \Q}V(Q)=V(\N)$,

\item 
$ E(I(F)) \subseteq F$ for every $F \in \es$, and 

\item $|I^{-1}(Q)|=|E(Q)|-1$ for every $Q\in \Q$.
\end{enumerate}

\end{definition}

Let $\ines(\R)=\F\setminus \es(\R)$ (the ``inessential'' sets) and  
 $B(\R)=\bigcup_{Q\in \Q}B(Q)$.

If such a regimentation $\R$ exists, we say then that $\F$ is regimented by $\R$.

Conditions (1) and (3) imply: 
\begin{lemma}\label{counting}
$|\es(\R)|=|V^{\circ}(\N)|$. 
\end{lemma}
\begin{proof}
Since $\es(\R)=\bigcup_{Q \in \Q} I^{-1}(Q)$, we have $|\es(\R)|=\sum_{Q \in \Q} |I^{-1}(Q)|$. Then by the condition (3) of a regimentation, we have
$$|\es(\R)|=\sum_{Q \in \Q} |I^{-1}(Q)|=\sum_{Q \in \Q} (|E(Q)|-1)=\sum_{Q \in \Q}|V^{\circ}(Q)|.$$
Since $\Q$ is a set of internally disjoint $s-t$ paths, the condition (1) of a regimentation implies $\sum_{Q \in \Q} |V^{\circ}(Q)|=|V^{\circ}(\N)|$, and hence we obtain $|\es(\R)|=|V^{\circ}(\N)|$.
\end{proof}

\begin{notation}[Pruning and concatenation of paths]\label{concatenatingpaths}
If $P$ is a directed path and $x \in V(P)$ then $Px$ is the part of $P$ up to and including $x$, and $xP$ is the part of $P$ starting at $x$. If two paths $P$ and $Q$ meet at a vertex $x$, then $PxQ$ denotes the walk obtained by concatenating $Px$ and $xQ$. 
If the endpoint of a path $P$ coincides with the initial point in a path $Q$, we write $PQ$ for the walk that is the concatenation of $P$ and $Q$. 
\end{notation}

\begin{lemma}\label{backward}
Suppose $\F$ is regimented by $\R=(\Q,I)$, and let $B=B(\R),  \ines=\ines(\R)$. If there is no $\F$-rainbow $s-t$ path, then $\bigcup   \ines \subseteq B$ and 
$\bigcup  I^{-1}(Q)\subseteq E(Q) \cup B \cup U(Q)$ for every $Q \in \Q$.
\end{lemma}

(For a set $\K$ of sets $\bigcup \K$ is the union of all sets  in $\K$.) 

\begin{proof}
Let $vu$ be an edge belonging to $F$ for some $F \in \F$.
Assume that $v \in V(Q_1), ~~u \in V(Q_2)$. Let $P=Q_1vuQ_2$ (see Notation \ref{concatenatingpaths}).

To obtain the conclusion of the lemma, we will show the following.
\begin{enumerate}
\item When $Q_1=Q_2$, $P$ is an $\F$-rainbow $s-t$ path unless $vu \in B(Q_1)$ or $vu \in E(Q_1)$ and $F \in I^{-1}(Q_1)$.
\item When $Q_1 \ne Q_2$, $P$ is an $\F$-rainbow $s-t$ path unless $v=t_{Q_1}$ and $F \in I^{-1}(Q_1)$, or $u=s_{Q_2}$ and $F \in I^{-1}(Q_2)$.
\end{enumerate}

First suppose that $Q_1=Q_2$.
If $v$ precedes $u$ on $Q_1$ and $vu \notin E(Q_1)$, then $P$ is an $\F$-rainbow $s-t$ path, since by part (3) of Definition \ref{reg} it has enough represented sets for its length.
If $vu \in E(Q_1)$, then $P$ is an $\F$-rainbow $s-t$ path unless $F \in I^{-1}(Q_1)$.
This proves (1).

Now assume $Q_1 \ne Q_2$.
 We may assume that $v \in V^{\circ}(Q_1)$ and $u \in V^{\circ}(Q_2)$ since if not 
the claim is a special case of (1). 
Then $Q_1v$ and $uQ_2$ are rainbow, and they have enough represented sets in $I^{-1}(Q_1)$ and $I^{-1}(Q_2)$, respectively.
If $F \notin I^{-1}(Q_1) \cup I^{-1}(Q_2)$, then $P$ is rainbow.
If $F \in I^{-1}(Q_1)$ and $v \ne t_{Q_1}$, then $Q_1vu$ is rainbow since it has enough represented sets in $I^{-1}(Q_1)$, since it has length at most $|E(Q_1)|-1$.
Similarly if $F \in I^{-1}(Q_2)$ and $u \ne s_{Q_2}$, then $vuQ_2$ is rainbow since it has enough represented sets in $I^{-1}(Q_2)$.
In both cases $P$ is rainbow, which  proves (2).

Since we assume there is no $\F$-rainbow $s-t$ path, if $F \in \ines$, then $vu \in B$ by (1) and (2).
Thus $\bigcup \ines \subseteq B$.
If $F \in I^{-1}(Q)$ for some $Q \in \Q$, then $vu \in E(Q) \cup B \cup U(Q)$ by (1) and (2).
Thus $\bigcup I^{-1}(Q) \subseteq E(Q) \cup B \cup U(Q)$.
\end{proof}

\begin{corollary}\label{nopath}
Let $\F$ be regimented by $\R$, and assume that there is no $\F$-rainbow $s-t$ path.
If $F \in \ines(\R)$ then $F$ does not contain an $s-t$ path. 
\end{corollary}
In fact, $F$ does not even contain an edge $sy$.

\begin{lemma}\label{lem.onlypath}
Let  $P, Q$ be $s-t$ paths in a network $(D,s,t)$. If $E(P) \subseteq E(Q) \cup B(Q) \cup \tb \cup U(Q)$ for some collection $\tb$ of edges that are  vertex-disjoint from $Q$, then $P=Q$.
\end{lemma}

\begin{proof}
    The only edge leaving $s$ in $E(Q)\cup B(Q) \cup \tb \cup U(Q)$ is $ss_Q \in E(Q)$, and the only edge to $t$ is $t_Q t \in E(Q)$. So these are necessarily the first and last edges of $P$. Therefore $P$ has no edges from $U(Q)$, since the in-degree of $s_Q$  and the out-degree of $t_Q$ 
    in $P$ are $1$. 
    
    As $E(Q) \cup B(Q)$ and $\tb$ are disconnected, $E(P) \cap \tb=\emptyset$. It remains to show that $E(P) \cap B(Q)=\emptyset$, which  follows from the fact that $P$ does not repeat vertices. 
\end{proof}

 Combining Lemmas \ref{backward} and \ref{lem.onlypath} yields:
 
 \begin{corollary}\label{onlypath}
 Let $\F$ be regimented by $\R$, and having no rainbow $s-t$ path.
 If $F \in \es(\R)$ then $I(F)$ is the only $s-t$ path contained in $F$.
\end{corollary}

By Corollaries \ref{nopath} and \ref{onlypath}, we can obtain the following corollary.

\begin{corollary}\label{iffcondition}
Let $\F$ be regimented by $\R$, and having no rainbow $s-t$ path.
Then $F \in \es(\R)$ if and only if $F$ contains an $s-t$ path, and equivalently, $F \in \ines(\R)$ if and only if $F$ does not contain an $s-t$ path.
\end{corollary}

 The following argument will be used twice, and hence it receives separate mention:
 
 \begin{lemma}\label{lem.exchange}
Let $\G, \ch$ be two   families of sets of edges, none of which possesses a rainbow $s-t$ path. Suppose that $\G$ is regimented by  $\R=(\Q,I)$ and $\ch$ is regimented by $\cs=(\cp,J)$. Suppose that  $\G \setminus \ch$  consists of a single set of edges $G$, and  $\ch \setminus \G$ consists of single set of edges $H$. Then either $G \in \ines(\R)$ and $H \in \ines(\cs)$, or 
 $I(G)=J(H)$.  
\end{lemma}

\begin{proof}
Let $\K=\G \cap \ch$. So $\G = \K \cup \{G\}, ~\ch=\K \cup \{H\}$.

By Corollary \ref{iffcondition}, it is obvious that 
\begin{equation}\label{sameessential}
    \K \cap \es(\R)= \K \cap \es(\cs).
\end{equation}

By  Corollary \ref{onlypath}, $I(K)=J(K)$ for every $K \in \K \cap \es(\R)$. Hence 
\begin{equation}\label{iandj}
    \bigcup_{K \in \es(\R)\setminus \{G\}}V(I(K))=\bigcup_{K \in \es(\cs)\setminus \{H\}}V(J(K)) \end{equation}

Let us first show  that $G \in \ines(\R)$ if and only if $H \in \ines(\cs)$. Suppose that  $G \in \ines(\R)$. Then $\es(\R) \subseteq \K$. By \eqref{sameessential} and Lemma \ref{counting}, it follows that $\es(\cs)=\es(\R)$, so $H \in \ines(\cs)$. The converse implication is the same. 

Assume next that $G\in \es(\R)$ and $H \in \es(\cs)$. Let $Q_0=I(G)$. Consider first the case that $V^\circ(Q_0)$ consists of a single vertex $v$.  We have  
$\bigcup_{K \in \es(R)\setminus \{G\}}V(I(K))=V^\circ \setminus \{v\}$, and hence by \eqref{iandj} we have also 
$\bigcup_{K \in \es(\cs)\setminus \{H\}}V(J(K))=V^\circ \setminus \{v\}$.
Since the interiors of paths in $\cp$ partition $V^\circ$, it follows that $J(H)$ is the path $svt$, namely $Q_0$. 

It remains to consider the case  $|V^\circ(Q_0)|>1$. Then, not counting multiplicities, $\cp=\Q$, because every path of $\Q$ appears as $J(K)$
for some $K \in \K$. The only path in $\cp$ not covered enough times by paths $J(K), ~K \in \es(\cs) \setminus\{H\}$, is $Q_0$. So, necessarily $J(H)=Q_0$. 
\end{proof}

The next theorem is the main step towards the proof of Theorem \ref{main}.

\begin{theorem}\label{thm: regimentedpaths}
Let $\N=(D,s,t)$ be a network with $n$ inner vertices. 
Let $\F$ be a family of 
$n+k-1$ sets of edges in  $\N$, satisfying the condition that $\bigcup\K$ contains an $s-t$ path, for every $\K \subseteq \F$ of size $k$. 
Then either there exists an $\F$-rainbow $s-t$ path, or $\F$ is regimented. 
\end{theorem}

The case $k=1$ of the theorem is Theorem 3.3 in \cite{akz}.

It is worth noting that the weaker result, with $\F$ being of size $n+k$, is not hard. First, the statement:

\begin{theorem}\label{thm: simplecase}
Let $\N=(D,s,t)$ be a network with $n$ inner vertices. 
Let $\F$ be a family of 
$n+k$ sets of edges in  $\N$, satisfying the condition that $\bigcup \K$ contains an $s-t$ path for every $\K \subseteq \F$ of size $k$. 
Then  there exists an $\F$-rainbow $s-t$ path. 
\end{theorem}

Here are two proofs:

{\em Proof 1}.  ~~
Observe that  a set $H$ of edges in $\N$ contains an $s-t$ path if and only if the cone of $\{\chi_{b}-\chi_{a} \mid ab \in H\}$ contains the vector $\chi_{t}-\chi_s$ (here $\chi_v$ is the function that is $1$ on $v$ and $0$ on all other vertices). Also note that all these vectors  reside in an $n+1$-dimensional space (they are of length $n+2$, but all are perpendicular to the all-$1$ vector). Apply now Theorem 
\ref{coopboth}, 
part (2). 

{\em Proof 2}.~~ Take a maximal $\F$-rainbow tree $T$ rooted at $s$. Assume, for contradiction, that it does not reach $t$. Then it represents at most $n$ members of $\F$. Hence there are $k$ sets $F \in \F$ not represented in $T$. By assumption, their union contains an $s-t$ path. The first edge leaving $T$ can then be added to $T$ to yield a larger rainbow tree, which contradicts the maximality of $T$.

\begin{definition}[contracting an edge $sx$] Let $sx$ be an edge of $\N$. We can contract $sx$ to a newly defined vertex $s'$, that will serve as the source of a new network $\N'$. Here is what this does to  sets of edges, and to paths. 

\begin{enumerate}
    \item Let $F$ be a set of edges in a network $\N=(D,s,t)$, and let  $sx$ be an edge, where $x$ is an inner vertex. The contracted   set of edges $F|_{sx \rightarrow s'}$ is obtained from $F$ by replacing  every edge $sy$ or $xy$ belonging to $F$ by the edge $s'y$, and removing all edges $yx$. 
    
    \item An $s-t$ path $P$ is transformed by the contraction of $sx$ to an $s'-t$ path $P'$, defined as follows. If  $x\not \in V(P)$ then $P'=P$ with $s'$ replacing $s$. If $x\in V(P)$ then $P'=s'yP$ where $y$ is the vertex following $x$ in $P$ (so, the vertices in $Px$ disappear.) We also write $P'=P|_{sx\rightarrow s'}$. Note that in this definition $E(P')$ is not necessarily equal to $E(P)|_{sx \rightarrow s'}$.
\end{enumerate}
\end{definition}

\begin{proof}[Proof of Theorem \ref{thm: regimentedpaths}.]
By induction on $n$.
The case $n=0$ is easy. So let $n \geq 1$ and assume that the theorem is valid when $n-1$ replaces $n$. 

Since $n+k-1 \geq k$, $\bigcup \F$ contains an $s-t$ path. So there exists at least one set $G\in \F$ containing an edge $sx$. If $x=t$ then the path $st$ is rainbow, so we may assume that $x \neq t$.
Now contract $sx$: for each $F \in \F$ let $F'=F|_{sx \rightarrow s'}$. 
Let $\K' =(F' \mid F \in \F)$ for $\K \subseteq \F$.
Let $\N'$ be the network with vertex set $V(\N)\setminus \{s,x\} \cup \{s'\}$,  source $s'$, target $t$, and edge set $\bigcup(\F' \setminus \{G'\})$.

Every $\K \subseteq \F$ of size $k$  contains in its union the edge set of an $s-t$ path in $\N$, which is easily seen to imply the same, with $s'$ replacing $s$, for $\K'$
in $\N'$. 
By the induction hypothesis, either there exists an $\F'\setminus \{G'\}$-rainbow $s'-t$ path $P'$, or $\F'\setminus\{G'\}$ is regimented. 
In the first case, let $y$ be the vertex following $s'$ in $P'$. Then either $syP'$ or $sxyP'$ is a rainbow $s-t$ path in $\N$, and we are done. 
So, we may assume the second possibility. 
Let $\R'=(\Q',I')$ be a regimentation of $\F'\setminus \{G'\}$, and let $\es'=\es(\R'), ~\ines'=\ines(\R')$.

Let $\tines=(F \in \F\setminus \{G\} \mid F' \in \ines')$ and  $\tes=(F \in \F\setminus \{G\} \mid F' \in \es')$.

 By Lemma \ref{counting}  $|\es'|=n-1$, so
 \begin{equation}\label{sizeinessential}
       |\tines|=|\ines'|=k-1.
 \end{equation}

In all claims below we assume  that there is no $\F$-rainbow $s-t$ path.

\medskip

 Let $B'=\bigcup_{Q' \in \Q'}B(Q')$.
By Lemma~\ref{backward}, $\bigcup \ines' \subseteq B'$ and $\bigcup I'^{-1}(Q') \subseteq E(Q') \cup B' \cup U(Q')$ for every $Q' \in \Q'$.

\begin{notation} [two ways of un-contracting $sx$]
Given an $s'-t$ path $Q'$  in $\N'$, let $Q'^{(1)}$ be the path obtained from $Q'$ by replacing $s'$ with $s$ and $Q'^{(2)}$  the path obtained from $Q'$ by expanding its first edge $s'y$ to the path $sxy$.
\end{notation}

Our aim is to glean from $\R'$ a regimentation $\R=(\Q, I)$ of $\F$. The set $\es(\R)$ will contain $G$ and $\Q$ will contain $s-t$ paths $f(Q'), ~~Q' \in \Q'$, where $f$ is an injective function defined as follows.
Let $Q' \in \Q'$ and let $F\in \F \setminus \{G\}$ be such that $I'(F')=Q'$. 
By \eqref{sizeinessential} and the condition of the theorem, the set $F \cup \bigcup \tines$ contains an $s-t$ path $Q$.
Let $f(Q')=Q$.

\begin{claim}
    
    $Q'=Q|_{sx\rightarrow s'}$.
    
\end{claim}
\begin{proof}
By the choice of $Q$, we have    $E(Q|_{sx\rightarrow s'}) \subseteq F' \cup \bigcup \ines'$. By Lemma  \ref{backward}, we have $F' \cup \bigcup \ines' \subseteq E(Q') \cup B' \cup U(Q')= E(Q') \cup B(Q') \cup \bigcup_{T' \in \Q' \setminus\{ Q'\}} B(T') \cup U(Q')$.
The claim now follows by  Lemma \ref{lem.onlypath}.
\end{proof}

There are two  possibilities:

\begin{enumerate}[($a$)]
    \item 
    $x \not \in V(Q)$. In this case $Q=Q'^{(1)}$.
    \item 
    $x \in V(Q)$. Suppose, in this case, that $Qx$ contains inner vertices. Let $y$ be the first inner  vertex of $Qx$. Then  $y \in V^\circ(T')$ for some $T' \in \Q' \setminus\{Q'\} $, and then $syT'$ is a rainbow $s-t$ path in $\N$ since it has enough represented sets in $I'^{-1}(T') \cup \{G\}$.
    So, we may assume that $V^\circ(Qx)=\emptyset$, meaning that the first edge on $Q$ is $sx$, meaning in turn that $Q=Q'^{(2)}$.
   \end{enumerate}

   \begin{claim}\label{sxie}
$sx \not \in \bigcup \tines$. 
\end{claim}

\begin{proof}
Let $F_0\in \tines$ and suppose that $sx \in F_0$. Recall that  
$\F'$ is the family of sets of edges obtained, where, for every $F \in\F$,  $F'$ is the image of $F$ under the contraction of $sx$.   
By the same  argument as above, 
$\F'\setminus \{F'_0\}$ is regimented in $\N'$, by a regimentation $\T=(\Q(\T), J)$. 
 Then $G' \in \ines(\T)$ by Lemma \ref{lem.exchange}, and hence $G$ do not contain an edge $yt$. But this would imply that   $G \bigcup \tines (\R)$ does not contain such an edge, and hence that it does not contain an $s-t$ path, contrary to the assumption of the theorem. 
\end{proof}

Since $E(Q) \subseteq F \cup \bigcup \tines$ and $\bigcup \ines' \subseteq B'$ by Lemma \ref{backward}, a corollary of Claim \ref{sxie} is: 
\begin{equation}\label{part1ofclaim}
    E(Q) \subseteq F.
\end{equation}

   \begin{claim} \label{onlyoneq} The choice of $f(Q')$ is independent of the choice of $F$.
   \end{claim}

   \begin{proof}
       We have to show that if $F_1, F_2 \in \F \setminus \{G\}$ satisfy  $I'(F_i')=Q', ~~i=1,2$ and 
       $Q_i$ are $s-t$ paths whose edge sets are contained in $F_i \cup \tines$ ($i=1,2$) then $Q_1=Q_2$. We know that $Q_i$ are either  $Q'^{(1)}$ or $Q'^{(2)}$. Assume, for contradiction, that $Q_1\neq Q_2$, say   $Q_1=Q'^{(1)}$ and  $Q_2=Q'^{(2)}$. Then $sx \in E(Q_2)$ and hence $sx\in F_2$. The set 
       $\F' \setminus\{F'_2\}$ lives in $\N'$, and repeating the previous argument we deduce that it has a regimentation $\cs=(\Q(\cs),J)$. By Lemma \ref{lem.exchange} $J(G')=I'(F_2')=Q'$. In particular $G' \supseteq E(Q')$. Since $Q_1=Q'^{(1)}$, the edge $ss_{Q'}$ belongs to $E(Q_1) \subseteq F_1$. Then, using an edge from $G$ and edges from  the sets $F\in \F$ such that  $F'\in I'^{-1}(Q')$ 
       shows that $ss_{Q'}Q'=Q'^{(1)}$ is an $\F$-rainbow $s-t$ path (note that edges in $E(s_{Q'} Q')$ are also   edges of $F$). This is the desired contradiction. 
          \end{proof}

\begin{claim}\label{modify}
\hfill
\begin{enumerate}
       \item If $f(Q')=Q'^{(2)}$ then $G \supseteq E(f(Q'))$.
       
       \item At most one $Q' \in \Q'$ satisfies $f(Q')=Q'^{(2)}$.

    \item If $f(Q')=Q'^{(1)}$ for all $Q' 
    \in \Q'$ then $G$ contains the edges of the $s-t$ path $sxt$.
\end{enumerate}
\end{claim}

\begin{proof}
To prove (1), let $f(Q')=Q'^{(2)}$ for some $Q' \in \Q'$.

Then, by Claim \ref{onlyoneq}, $sx \in F$ for every $F' \in I'^{-1}(Q')$.
We use the same trick as in the proof of Claim \ref{onlyoneq},  interchanging   the roles of $F$ and $G$. Consider $\F' \setminus \{F'\}$.
As above, we may assume that $\F' \setminus \{F'\}$ is regimented, by a regimentation  $(\cp',J')$.
By Lemma \ref{lem.exchange},  $J'(G')=I'(F')=Q'$, implying that $G' \supseteq E(Q')$.
Then $G$ contains either $E(Q'^{(1)})$ or $E(Q'^{(2)})$.
If $G$ contains $E(Q'^{(1)})$, then $ss_{Q'}Q'$ (which is just $Q'^{(1)}$) is an $\F$-rainbow $s-t$ path: the edge $ss_{Q'}$ represents  $G$; since $|I'^{-1}(Q')|=|E(Q')|-1$, the other edges have enough represented sets $F \in \F$ such that $F' \in I'^{-1}(Q')$ (remember that $G \not \in I'^{-1}(Q')$). 
We have thus shown that $G$ does not contain $E(Q'^{(1)})$, so it contains $E(Q'^{(2)})$, namely  $G \supseteq E(f(Q'))$.

Next we prove (2). 
Let $f(Q')=Q'^{(2)}$ for some $Q' \in \Q'$.
By the above argument and Corollary \ref{onlypath}, $J'(G')=Q'$ is the only path contained in $G'$.
This directly implies (2).

Finally, we prove (3). Assume that  $f(Q')=Q'^{(1)}$ for all $Q' 
    \in \Q'$. 
Let $\tilde{\N}$ be the network obtained from $\N$ by deleting the vertex $x$, and let $\tilde{F}$ be the  set of edges of $\tilde{\N}$, obtained from $F$ by deleting all edges incident with $x$.
Let
$\tilde{\Q}=\{Q'^{(1)} \mid Q' \in \Q'\}$, and $\tilde{I}(\tilde{F})=f(I'(F'))$. By 
\eqref{part1ofclaim} and the assumption that $f(Q')=Q'^{(1)}$ for all $Q' \in \Q'$ the set $\tilde{\F}=(\tilde{F} \mid F \in \F)$ is regimented by $(\tilde{\Q},\tilde{I})$.
The fact that there is no $\F$-rainbow $s-t$ path implies that  there is also  no $\tilde{\F}$-rainbow $s-t$ path. Therefore,  by Lemma \ref{backward}, we have $\tilde{G} \cup \bigcup_{F \in \tilde{\ines}} \tilde{F} \subseteq \bigcup_{Q \in \tilde{\Q}} B(Q)$.
Thus 
\[G \cup \bigcup \tines \subseteq \{sx,xt\} \cup \bigcup_{Q' \in \Q'} B(Q'^{(1)}) \cup U(sxt).\]
By the assumption of the theorem, $G \cup \bigcup \tines$ contains an $s-t$ path, say $Q_G$. By Lemma \ref{lem.onlypath} we have $Q_G=sxt$, and by Claim \ref{sxie} we obtain $G \supseteq E(Q_G)$. 
 This concludes the proof of the claim. \end{proof}

\begin{remark}\label{disjoint}
By the claim  the paths $f(Q'), ~~Q' \in \Q'$ are internally disjoint. In particular, there is at most one path $f(Q')$ containing $x$.  
\end{remark}

   We can now complete the induction step in the proof of Theorem \ref{thm: regimentedpaths}, by showing that $\F$ is regimented. \\ 
   
{\bf Case I:}~~ $f(Q')=Q'^{(1)}$ for all $Q' \in \Q'$.

Let $\Q=\{f(Q') \mid Q' \in \Q'\} \cup \{Q_0\}$ where $Q_0=sxt$. 
Let $\es=( F \mid F' \in \es(\R') ) \cup \{G\}$.
Define $I: \es \to ~\Q$  by $I(F)=f(I'(F'))$ for $F \neq G$, and $I(G)=Q_0$. 

\begin{claim}\label{regimentationcase1}
$(\Q,I)$ is a regimentation of $\F$.
\end{claim}

By Remark  \ref{disjoint} and the fact that $x \notin \bigcup_{Q' \in \Q'} V(f(Q'))$, $\Q$ is a set of internally disjoint $s-t$ paths.

By \eqref{part1ofclaim} $E(I(F)) \subseteq F$ for all $F \in \es \setminus \{G\}$, and by part (3) of Claim \ref{modify} $E(I(G))=E(Q_0) \subseteq G$.
This implies condition (2) in Definition \ref{reg}.

In addition, 
\[ |I^{-1}(Q)|=|I'^{-1}(f^{-1}(Q))|=|E(f^{-1}(Q))|-1=|E(f^{-1}(Q)^{(1)})|-1
=|E(Q)|-1
\]
for all $Q \in \Q \setminus \{Q_0\}$, and 
\[ |I^{-1}(Q_0)|=1=|E(Q_0)|-1.\]
This yields condition (3) of Definition \ref{reg}.

Furthermore, 
since $\bigcup_{Q' \in \Q'}V^{\circ}(Q')=V^{\circ}(\N)\setminus\{x\}$ and $V^{\circ}(Q'^{(1)})=V^{\circ}(Q')$, we have 
\[
\bigcup_{Q \in \Q}V^{\circ}(Q)=\bigcup_{Q' \in \Q'}V^{\circ}(Q'^{(1)}) \cup \{x\}=V^{\circ}(\N).
\]
This implies  condition (1) of Definition \ref{reg}, thus  completing the proof of the claim.

{\bf Case II:}~~
$f(Q_0')=Q_0'^{(2)}$ for some $Q_0' \in \Q$.

Let $\Q=\{f(Q') \mid Q' \in \Q'\}$ and 
$\es =(F \mid F' \in \es(\R'))  \cup \{G\}$.
Define $I : \es \to \Q$ by $I(F)=f(I'(F'))$ for all $F \in \F \setminus \{G\}$
and $I(G)=f(Q_0')$. 

\begin{claim}\label{regimentationcase2}
 $(\Q,I)$ is (here, too) a regimentation of $\F$.
\end{claim}

By Remark \ref{disjoint}, $\Q$ is a set of internally disjoint $s-t$ paths.

By \eqref{part1ofclaim} $E(I(F)) \subseteq F$ for $F \in \es \setminus \{G\}$,   and by (1) of Claim \ref{modify} $E(I(G))=E(f(Q_0')) \subseteq G$, so condition (2) of Definition \ref{reg} is fulfilled.

In addition,
\[ |I^{-1}(Q)|=|I'^{-1}(f^{-1}(Q))|=|E(f^{-1}(Q))|-1=|E(f^{-1}(Q)^{(1)})|-1=|E(Q)|-1\]
for all $Q \neq f(Q_0')$. On the other hand, for  $Q=f(Q_0')$, 
\[
|I^{-1}(Q)|=|I'^{-1}(f^{-1}(Q))|+1=|E(f^{-1}(Q))|=|E(f^{-1}(Q)^{(2)})|-1=|E(Q)|-1.\]
This proves condition (3) in Definition \ref{reg}.

Furthermore, since $\bigcup_{Q' \in \Q'}V^{\circ}(Q')=V^{\circ}(\N)\setminus\{x\}$, $V^{\circ}(Q'^{(1)})=V^{\circ}(Q')$ and $V^{\circ}(Q'^{(2)})=V^{\circ}(Q') \cup \{x\}$, we have 
\[
\bigcup_{Q \in \Q}V^{\circ}(Q)=\bigcup_{Q' \in \Q' \setminus\{Q_0'\}} V^{\circ}(Q'^{(1)}) \cup V^{\circ}(Q_0'^{(2)})=V^{\circ}(\N).
\]

 So, condition (1) of Definition \ref{reg} is also valid,  completing the proof of the theorem.  
\end{proof}

\section{Proof of Theorem \ref{main}}
Let us first state the theorem in a slightly stronger form, that allows some of the edge sets to be empty.

\begin{theorem}
Let $\cs$ be a family  of  $2n+k-3$ sets of edges in a bipartite graph $G$,  at most $k-2$ of them being empty. If  $\nu(\bigcup \K) \ge n$ for every $\K \subseteq \cs$ of size $k$ then $\nu_R(\cs)\ge n$.
\end{theorem}

Before proving the theorem, we need the following definition.
\begin{definition}
For a matching $N$ in a graph, a path is called {\em $N$-alternating} if every other edge in it belongs to $N$ and it is called {\em augmenting} if its starting edge and ending edge are not in $N$. 
\end{definition}

   \begin{proof}
Suppose, for contradiction, that $\nu_R(\cs) =: m<n$. Let $M=\{f_S \mid S \in \cs_0\}$ be  a maximal size $\cs$-rainbow matching, 
where $f_S \in S$. Let $\cs_0^c=\cs \setminus \cs_0$. 

Let $A,B$ be the two sides of $G$.
For every  $h\in E(G)$ let $h_A$ be the $A$-vertex of $h$, and $h_B$ the $B$ vertex. 

We construct a network $\N$, having the property that its paths correspond to $M$-alternating paths, and its source-target paths correspond to augmenting $M$-alternating paths. Let $V(\N)=M \cup \{s,t\}$, 
where $s$ represents $U_A:=A\setminus \bigcup M$, and $t$ represents $U_B:=B\setminus \bigcup M$. 

To every edge $h=ab\in E(G) \setminus M$ ($a \in A, b \in B$) we assign an edge $F(h)$ of $\N$, as follows. 

\begin{enumerate}
    \item If $a \in f \in M, ~b \in g \in M$ then $F(h)=fg$.
    
    \item 
If $a \in U_A$ and $b \in g \in M$ then $F(h)=sg$. 

\item 
If $b \in U_B$ and $a \in f \in M$ then $F(h)=ft$. 

\item If $a \in U_A$ and $b \in U_B$  then $F(h)=st$.
\end{enumerate}

For a set $S$ of edges in $G$, let $F(S)$ be the set of edges in $\N$, defined by $F(S)=\{F(h) \mid h \in S \setminus M\}$.
The function $F$ is not one-to-one, because the inverse image of an edge $sh$ ($h \in M$) can be any edge $ah_B$, $a \in U_A$.

Clearly, if $M\cup S$ contains an augmenting $M$-alternating path, then $F(S)$ contains an $s-t$ path in $\N$, and vice versa. 
Let $\F=\{F(S) \mid S \in \cs_0^c\}$.

 Since,  by assumption, $m<n$, $|\cs_0^c|=2n-m+k-3 \ge m+k-1$. 
 If $N$ is a matching of size $n$, then $M \cup N$ contains an augmenting $M$-alternating path, and hence $F(N)$ contains an $s-t$
path. 
Hence, by Theorem  
\ref{thm: regimentedpaths} and Theorem \ref{thm: simplecase}, either

(i) there exists an $\F$-rainbow $s-t$ path $P$, or 

(ii) $|\cs_0^c|= m+k-1$ and $\F$ is regimented.

In  case (i), as mentioned above, $P$ yields an augmenting $M$-alternating path, whose application yields a larger rainbow matching. So  we may assume (ii).
Let $\R=(\Q,I)$ be the regimentation of $\F$.
Let $F^{-1}(\ines(\R))=(S \in \cs_0^c \mid F(S) \in \ines(\R))$.
Since at most $k-2$ sets $S \in \cs$ are empty and $|\ines(\R)|=|\cs_0^c|-|\es(\R)|=k-1$ by Lemma \ref{counting}, $\bigcup F^{-1}(\ines(\R))$ is non-empty.

\begin{claim}\label{nonemptyines}
   It is possible to choose $M$ so that  $\bigcup \ines(\R) \ne \emptyset$.
\end{claim}

This means that $\bigcup F^{-1}(\ines(\R)) \setminus M \ne \emptyset$.

\begin{proof} 
Assume, for contradiction, that $\bigcup F^{-1}(\ines(\R)) \subseteq M$.
Since $\bigcup F^{-1}(\ines(\R))$ is non-empty, there is an element $S_0 \in \cs_0$ such that $f_{S_0} \in M \cap \bigcup F^{-1}(\ines(\R))$.
Let $S_1$ be a set in $F^{-1}(\ines(\R))$ containing $f_{S_0}$.
 By the condition of the theorem, $\bigcup F^{-1}(\ines(\R)) \cup S_0$ contains a matching of size $n$. This, in turn,  means that there exists an edge $f \in \bigcup F^{-1}(\ines(\R)) \cup S_0 \setminus M$.
Since by assumption $\bigcup F^{-1}(\ines(\R)) \subseteq M$, we have $f \in S_0$.
Now we can consider $\cs_1=(\cs_0 \setminus \{S_0\}) \cup \{S_1\}$ as a represented set of $M$ by changing the roles of $S_0$ and $S_1$.
Let $\tilde{\F}=(F(S) \mid S \in \cs_1^c)$.
Then by the same reasoning as  above, we may assume that $\tilde{\F}$ is regimented by $\tilde{\R}=(\tilde{\Q},\tilde{I})$.
By Lemma \ref{lem.exchange}, we have $F(S_0) \in \ines(\tilde{\R})$ and $f \in S_0 \setminus M$, which implies $\bigcup \ines(\tilde{\R}) \ne \emptyset$.
\end{proof}

So,  we  assume $\bigcup \ines(\R) \ne \emptyset$.
Let $pq$ be an edge in $F(S)$ for some $F(S) \in \ines(\R)$.
By Lemma \ref{backward}, $pq$ is a backward edge on some path $Q \in \Q$.
Let $Q=sy_1y_2\ldots y_ct$. 
For each $1\le i < c$  let $e_i$ be the edge connecting the $(y_i)_A$ with $(y_{i+1})_B$, in $G$ (these are the $F^{-1}$ images of the edges of $Q$). 

Let $\ell$ be such that $p=y_\ell$. As $p$ is an edge in $M$, $p$ is contained in a set $S_p \in \cs_0$. By the condition of the theorem, the set $S_p \cup \bigcup F^{-1}(\ines(\R))$ contains a matching $N$ of size $n$. 
Since $|M|<n$,  $N$ contains an edge $ax$, where $a \in U_A$ (recall that $U_A= A \setminus \bigcup M$).
Suppose $x \in U_B$.
If $ax \in \bigcup F^{-1}(\ines(\R))$, then $M \cup \{ax\}$ is a rainbow matching, contradicting the maximality of $M$. Thus we have $ax \in S_p$.
Let $q=y_{\ell'}$ for some $\ell'<\ell$.
Now consider $$N=(M \cup \{ax,p_A q_B\}\cup\{(y_{i})_A (y_{i+1})_B \mid \ell' \leq i \leq \ell-1\})\setminus \{y_{\ell'},y_{\ell'+1},\ldots,y_{\ell}\}.$$
Since $p_A q_B \in S$ and $\{(y_{i})_A (y_{i+1})_B \mid \ell' \leq i \leq \ell-1\}$ has enough represented sets in $I^{-1}(Q)$, then $N$ is a rainbow matching.
However, it is a contradiction to the maximality of $M$ since $N$ has size $|M|+1$.

Hence, we may assume that $x$ lies on an edge $h$ of $M$, meaning that $sh$ is an edge in $F(S_p) \cup \bigcup \ines(\R)$. 
Since all edges in $\bigcup \ines(\R)$ are backwards, and $sh$ is not a backward edge on any path, $sh$ belongs to $F(S_p)$.
 
 Let $h \in V(Q_h)$ for  $Q_h \in \Q$, and let $P$ be the $s-t$ path $shQ_h$.
 Let $\tilde{P}$ be a path in $F^{-1}(P)$, whose first vertex is $a$, meaning that its first edge belongs to $S_p$. 
 Let $X \triangle Y$ be the symmetric difference of $X$ and $Y$, that is, $X \triangle Y=(X \setminus Y) \cup (Y \setminus X)$.
 Let $N=M\triangle E(\tilde{P})$.

 Consider  two possibilities: 

{\bf Possibility I}:
$h =y_d$ for $d \le \ell$. 

In this case  
$N$ is an $\cs$-rainbow matching of size $m+1$: we  let the first edge, $ah_B$, represents $S_p$,
and the other edges in $E(\tilde{P}) \setminus M$ has a represented sets in $I^{-1}(Q)$
and 
keep all other representations as they are. Since the edge in $M$ representing $S_p$ is removed by the symmetric difference, this assignment of representation yields an $\cs$-rainbow matching.

{\bf Possibility II}: Either $h \notin V(Q)$ or $h =y_d$ for $d >\ell$. 

In this case, $N$ is not $\cs$-rainbow, since there are two edges representing $S_p$, namely $p$ and $ah_B$. But this is rectifiable, using the edge $pq$. Suppose that $q=y_b$, where $b<\ell$. Let 
$C$ be the cycle whose edges are $p_Aq_B,q, e_b, y_{b+1}, e_{b+1}, \ldots ,e_{\ell-1}, p=y_\ell$. Let $N'=N\triangle E(C)$. Then $N'$ is a matching of size $m+1$, and it is $\cs$-rainbow, since $S_p$ is represented in it just once - by the edge $ah_B$.
\end{proof}

\section{Somewhere over the rainbow - two possible strengthenings} 

It is possible that Theorem \ref{main} can be given  a strong cooperation generalisation. 

\begin{conjecture}\label{halldrisko}
Let $\F$ be  a family of $2k-1$ sets of edges in a bipartite graph. If  $\nu(\bigcup \K) \ge \min(|\K|,k)$ for every $\K \subseteq \F$  then $\nu_R(\F)\ge k$.
\end{conjecture}

This  generalises  the following theorem from 
\cite{abkz}:

\begin{theorem}\label{akz}
If $\F=(F_1, \ldots ,F_{2k-1})$ is a family of matchings in a bipartite graph, and $|F_i|=\min(i, k)$ for all $i$, then there exists an $\F$-rainbow matching of size $k$. 
\end{theorem}

Here is another possible strong version of Theorem \ref{main}.

\begin{conjecture}\label{double}
Let $\F=(F_1, \ldots , F_{2k-1})$ be  a system of bipartite sets of edges, sharing the same bipartition, and suppose that $\nu(F_i)\ge k$ for all $i \le 2k-1$. Let $V'$ be a copy of $V$ disjoint from $V$,  let  $F'_i $ be a copy of $F_i$ on $V'$ ($i \le 2k-1$)  and let $\tilde{F}_i=F_i \cup F'_i$ for $i \le 2k-1$. Then the system $(\tilde{F}_i \mid i \le 2k-1)$ has a full rainbow matching. 
\end{conjecture}

This implies Theorem \ref{drisko}, since by the pigeonhole principle either $V$ or $V'$ contains a rainbow matching of size $k$. Conjecture \ref{double} would follow from the following conjecture of the first author and Eli Berger \cite{ab}. 

\begin{conjecture}
$n$ matchings of size $n$ in any graph have a rainbow matching of size $n-1$.
\end{conjecture}

\subsection*{Acknowledgement}
This paper is a part of a project that has received funding from the European Union's Horizon 2020 research and innovation programme under the Marie Sk$\ell$lodowska-Curie grant agreement no. 823748. This work was supported by the Russian Federation Government in the framework of MegaGrant no. 075-15-2019-1926 when R. Aharoni worked on Sections 1, 2 and 3 of the paper.
The research of R. Aharoni was supported by the Israel Science Foundation (ISF) grant no. 2023464 and the Discount Bank Chair at the Technion. 
The research of J. Briggs was supported by ISF Grant No. 326/16, ISF Grant No. 1162/15, and ISF Grant No. 409/16.
The research of M. Cho was supported by the National Research Foundation of Korea (NRF) grant funded by the Korea government(MSIT) (No. NRF-2020R1F1A1A01048490).
The research of J. Kim was supported by BSF Grant No. 2016077, ISF Grant No. 1357/16 and the Institute for Basic Science, Republic of Korea (IBS-R029-C1).
This work was completed while J. Briggs and J. Kim were post-doctoral research fellows at the Technion.

The authors thank the anonymous referees for their helpful comments.

\end{document}